\documentclass[a4paper,11pt]{amsart}
\usepackage{amsmath}
\usepackage{amsfonts}
\usepackage{amssymb}
\usepackage{graphicx}
\usepackage[abbrev,alphabetic]{amsrefs}
\usepackage{color}
\usepackage{soul,xcolor} 
\setstcolor{red}
\definecolor{red}{rgb}{0.6,0,0}

\usepackage{amsthm}
\usepackage{comment}
\usepackage[all,cmtip]{xy}
\usepackage{tikz-cd}
\usetikzlibrary{cd}

\newcommand{\stacksproj}[1]{{\cite[Tag~{#1}]{stacks-project}}}
\newcommand{\Fp}{\mathbb{\overline{F}}_p}

\newcommand{\mbQ}{\mathbb{Q}}

\newcommand{\mbR}{\mathbb{R}}

\newcommand{\mbZ}{\mathbb{Z}}

\newcommand{\mcA}{\mathcal{A}}

\newcommand{\mcF}{\mathcal{F}}
\newcommand{\mcG}{\mathcal{G}}

\newcommand{\mcI}{\mathcal{I}}

\newcommand{\mcO}{\mathcal{O}}

\DeclareMathOperator{\Supp}{Supp}
\DeclareMathOperator{\Spec}{Spec}
\DeclareMathOperator{\codim}{codim}

\DeclareMathOperator{\Pic}{Pic}

\DeclareMathOperator{\Exc}{Exc}

\newcommand*{\coloneq}{\mathrel{\mathop:}=}
\newcommand*{\defeq}{\mathrel{\mathop:}=}

\theoremstyle{plain}
\newtheorem{theorem}{Theorem}[section]
\newtheorem{proposition}[theorem]{Proposition}
\newtheorem{lemma}[theorem]{Lemma}
\newtheorem{corollary}[theorem]{Corollary}

\theoremstyle{definition}
\newtheorem{definition}[theorem]{Definition}

\theoremstyle{remark}
\newtheorem{remark}[theorem]{Remark}

\title[On the relative MMP in low characteristics]{On the relative Minimal Model Program for threefolds in low characteristics}
\author{Christopher Hacon}
\address{Department of Mathematics \\
University of Utah\\
Salt Lake City, UT 84112, USA}
\email{hacon@math.utah.edu}\email{}

\author{Jakub Witaszek}
\address{Department of Mathematics, University of Michigan, Ann Arbor, MI 48109, USA} 
\email{jakubw@umich.edu}

\begin{document}

\begin{abstract}
We show the validity of the relative dlt MMP over $\mbQ$-factorial threefolds in all characteristics $p>0$. As a corollary, we generalise many recent results to low characteristics including: $W\mcO$-rationality of klt singularities, inversion of adjunction,  and normality of divisorial centres up to a universal homeomorphism.
\end{abstract}


\maketitle

\section{Introduction}

One of the fundamental tools used in the study of algebraic varieties is the \emph{Minimal Model Program} (MMP), which is a higher-dimensional analogue of the classification of surfaces. This program has been a topic of intense research in the last decades and a major part of it is known to hold for projective varieties defined over a field of characteristic zero (see \cite{bchm06}).

Moreover, in recent years, there has been substantial progress in the MMP in positive characteristic. In particular, it has been shown that the program is valid for surfaces over excellent base schemes (see \cite{tanaka12,tanaka16_excellent}) and for three-dimensional varieties defined over perfect fields of characteristic $p>5$ {(see \cite{hx13,ctx13,birkar13,BW14,GNT06})}. This has led to many striking applications, for example, Gongyo, Nakamura, and Tanaka used the MMP to show the existence of rational points on log Fano threefolds defined over finite fields $\mathbb{F}_{p^n}$ for $p>5$.

The goal of this article is to extend some of the foundational results on the MMP for threefolds, and applications thereof, to low characteristics. Our main result is the following.

\begin{theorem} \label{thm:main} Let $(Y,\Delta)$ be a three-dimensional $\mbQ$-factorial dlt pair defined over a perfect field of characteristic $p>0$. Assume that there exists a projective {birational} morphism $\pi \colon Y \to X$ over a normal $\mbQ$-factorial variety $X$ such that $\Exc(\pi) \subseteq \lfloor \Delta \rfloor$. Then $\pi$-relative contractions and flips exist for $K_Y+\Delta$, and we can run a $(K_Y+\Delta)$-MMP over $X$ which terminates with a minimal model.
\end{theorem}
Theorem \ref{thm:main} allows us to study  three-dimensional singularities by means of the MMP. In particular, it can be used to construct Koll\'ar components. 

The proof of the existence of flips for threefolds in positive characteristic $p>5$ relies on the fact that pl-flipping contractions are purely relatively F-regular. Although this is not true in general when $p \leq 5$, we show that pl-flipping contractions occurring in the relative dlt-MMP over $\mbQ$-factorial threefolds, as in the theorem above, do satisfy this property. Therefore, we can make use of the same strategy as in \cite{hx13} to conclude the construction of such pl-flips.

In order to show the existence of pl-contractions, it was proven in \cite[Proposition 4.1]{hx13} that divisorial centres of three-dimensional plt pairs are normal in characteristic $p>5$, so the main theorem of \cite{keel99} could be invoked. However, this result about divisorial centres is false in characteristic { $p=2$, in general, as shown in \cite{CT06PLT} (cf.\ \cite{bernasconi18})}. Nevertheless, Gongyo, Nakamura, and Tanaka observed that in order to use \cite{keel99}, it is enough to show that divisorial centres are normal up to a universal homeomorphism; moreover, they proved that this property of divisorial centres is stable under the MMP. Therefore, they were able to show that the existence of pl-flips in low characteristic implies the existence of minimal models (cf.\ \cite[Theorem 3.15]{GNT06}). The same strategy concludes the proof of Theorem \ref{thm:main} and implies the validity of the following results (cf.\ \cite[Theorem 3.14]{GNT06}).

\begin{theorem}[Normality of divisorial centres up to a universal homeomorphism] \label{thm:normality} Let $(X,S+B)$ be a $\mbQ$-factorial three-dimensional plt pair defined over a perfect field of characteristic $p>0$, where $S$ is an irreducible divisor. Then the normalisation $f \colon \tilde S \to S$ is a universal homeomorphism.
\end{theorem}
\noindent In fact, we show that all log canonical centres of $\mbQ$-factorial three-dimensional dlt pairs are normal up to a universal homeomorphism (see Remark \ref{rem:smaller_centres_normal}).

Furthermore, we obtain the following applications of Theorem \ref{thm:main}.
\begin{corollary}[{$W\mcO$-rationality, cf.\ \cite[Theorem 3.16]{GNT06}}] \label{cor:wittrationality} Let $X$ be a klt three-dimensional $\mbQ$-factorial  variety defined over a perfect field of characteristic $p>0$. Then $X$ has $W\mcO$-rational singularities.
\end{corollary}

\begin{corollary}[dlt modification] \label{cor:dlt_mod} Let $(X,\Delta)$ be a three-dimensional $\mbQ$-factorial  log pair defined over a perfect field of characteristic $p>0$. Then a dlt modification of $(X,\Delta)$ exists, that is, a birational morphism $\pi \colon Y \to X$ such that $(Y, \pi^{-1}_*\Delta + \Exc(\pi))$ is dlt, $\mbQ$-factorial, and minimal over $X$.
\end{corollary} 

\begin{corollary}[{Inversion of adjunction, cf.\ \cite[Theorem A]{Das15}}] \label{cor:plt} Consider a $\mbQ$-factorial three-dimensional log pair $(X,S+B)$ defined over a perfect field of characteristic $p>0$, where $S$ is an irreducible divisor. Then $(X,S+B)$ is plt on a neighborhood of $S$ if and only if $(\tilde S,B_{\tilde S})$ is klt, where $\tilde S$ is the normalisation of $S$ and $B_{\tilde S}$ is the different.
\end{corollary}

Moreover, using Theorem \ref{thm:main}, we may generalise some other results previously known only when $p>5$.
\begin{itemize}
	\item divisorial contractions of extremal rays exist in the category of $\mbQ$-factorial three-dimensional varieties for $p>0$,
	\item the tame fundamental group of three-dimensional $\mbQ$-factorial klt singularities is finite when $p>0$, and
	\item minimal log canonical centres of $\mbQ$-factorial lc pairs $(X,\Delta)$, with $X$ being klt, are normal up to a universal homeomorphism for $p>0$.
\end{itemize}

Lastly, we show the validity of the $(K_X+\Delta)$-MMP in the setting of dlt reductions. In particular, this allows for the construction of canonical skeletons of degenerations of varieties with non-negative Kodaira dimension (see \cite{NX16}). {Here $R$ is a local ring of a curve $C$ defined over a perfect field $k$ of characteristic $p>0$. We denote the special point $s \in \Spec R$ and the generic point $\eta \in \Spec R$.  Let $(\mathcal{X}, \Phi)$ be a $\mbQ$-factorial dlt pair of dimension three and projective over $C$ and let  $(X,\Delta):=(\mathcal{X}, \Phi)\times _C {\rm Spec}(R)$.

\begin{theorem} \label{thm:mmp_reductions} Let $(X,\Delta)$ be a  dlt pair as above so that in particular $X$ is three-dimensional $\mbQ$-factorial projective over $\Spec R$. Let $\phi \colon X \to \Spec R$ be the natural projection and suppose that $\Supp \phi^{-1}(s) \subseteq \lfloor \Delta \rfloor$. Then we can run a $(K_X+\Delta)$-MMP over $\Spec R$ which terminates with a minimal model or a Mori fibre space.
\end{theorem} 

This article is organised as follows. In Section 2 we gather preliminary results on F-regularity, the MMP in positive characteristic, and $W\mcO$-rationality. In Section 3 we show Theorem \ref{thm:main}, Corollary \ref{cor:wittrationality}, Corollary \ref{cor:plt}, and Corollary \ref{cor:dlt_mod}. In Section 4 we show Theorem \ref{thm:mmp_reductions}. Section 6 pertains to the study of other applications of the main theorem.

\section{Preliminaries}
A scheme $X$ will be called a variety if it is integral, separated, and of finite type over a field $k$. Throught this paper, $k$ is a perfect field of characteristic $p>0$. We refer to \cite{km98} for basic definitions in birational geometry, to \cite{bcekprsw02} for the theory of nef reduction maps and nef dimension, and to \cite{kollar96} for results pertaining to MRCC (maximally-rationally-chain-connected) fibrations. 
We remark that in this paper, unless otherwise stated, if $(X,B)$ is a pair, then $B$ is a $\mathbb Q$-divisor.

We advise the reader to consult \cite[Remark 2.7]{GNT06} in regards to the subtleties of birational geometry over a perfect field. Let us just recall that being klt, plt, lc, or normal is preserved under taking a base change to an uncountable algebraic closure. Similarly for the ampleness, semi-amplenes, nefness and bigness of a line bundle. The same also holds for rational chain connectedness (see \cite[Remark 4.6]{GNT06}). However, this is not the case with $\mbQ$-factoriality.  

\subsection{Relative F-regularity}
The study of relative F-regularity plays a vital role in the proof of the existence of flips for threefolds in characteristic $p>5$, acting as a replacement for vanishing theorems (see \cite{hx13}). 

For the convenience of the reader, we recall the basic definitions and results. This section is based on \cite{HW17}.
\begin{definition} For an F-finite scheme X of characteristic $p>0$ and an effective $\mbQ$-divisor $\Delta$, we say that $(X,\Delta)$ is \emph{globally F-split} if for every $e \in \mbZ_{>0}$, the natural morphism
\[
\mcO_X \to F^e_* \mcO_X(\lfloor(p^e-1)\Delta \rfloor)
\]
splits in the category of sheaves of $\mcO_X$-modules.

We say that $(X,\Delta)$ is \emph{globally F-regular} if for every effective divisor $D$ on $X$ and every big enough $e \in \mbZ_{>0}$, the natural morphism
\[
\mcO_X \to F^e_* \mcO_X(\lfloor(p^e-1)\Delta \rfloor + D)
\]
splits.

We say that $(X,\Delta)$ is \emph{purely globally F-regular}, if the definition above holds for those $D$ which intersect $\lfloor \Delta \rfloor$ properly (see \cite{Das15}*{Definition 2.3(2)}).
\end{definition}
The local versions of the above notions are called F-purity, strong F-regularity, and pure F-regularity, respectively. Moreover, given a morphism $f \colon X \to Y$, we say that $(X,\Delta)$ is F-split, F-regular, and purely F-regular over $Y$, if the corresponding splittings hold locally over $Y$ (see \cite{hx13}*{Definition 2.6}).

Let us recall that the relative inversion of F-adjunction holds.
\begin{lemma}[\cite{Das15} and {\cite[Lemma 2.10]{HW17}}] \label{lemma:relative-inversion-of-adjunction}
Let $(X,S+B)$ be a plt pair where $S$ is a prime divisor, and let $f \colon X \to Z$ be a proper birational morphism between normal varieties { defined over a perfect field of characteristic $p>0$}. Assume that $-(K_X+S+B)$ is $f$-ample and $(\bar{S},B_{\bar{S}})$ is globally $F$-regular over $f(S)$, where $\bar{S}$ is the normalization of $S$, and $B_{\bar{S}}$ is defined by adjunction $K_{\bar{S}} + B_{\bar{S}} = (K_X + S + B)|_{\bar{S}}$. Then $(X,S+B)$ is purely globally F-regular over a Zariski-open neighbourhood of $f(S) \subseteq Z$. \end{lemma}

Furthermore, divisorial centres of purely F-regular pairs are normal.
\begin{proposition}  \label{proposition:normality_of_plt_centres} Let $(X,S+B)$ be a purely F-regular pair defined over a perfect field $k$ of characteristic $p>0$ where $S$ is a prime divisor. Then $S$ is normal.
\end{proposition}
\begin{proof}
This follows from \cite[Theorem A]{Das15} by taking a base change to the algebraic closure of $k$.
\end{proof}

\subsection{MMP for threefolds in arbitrary positive characteristic}
In this subsection we gather some results on the Minimal Model Program for threefolds which are valid in any characteristic $p>0$.

Firstly, the cone theorem is known for pseudo-effective adjoint divisors.
\begin{theorem}[{\cite[Lemma 2.2]{GNT06}}] \label{thm:cone}

Let $(X,\Delta)$ be a three-dimensional $\mbQ$-factorial log canonical pair which is projective over a separated scheme $U$ of finite type over a perfect field $k$ of characteristic $p>0$. Assume that $K_X+\Delta$ is pseudo-effective over $U$. Then, for every ample $\mbQ$-divisor $A$, there exist finitely many curves $C_1, \ldots, C_r$ such that
\[
\overline{\mathrm{NE}}(X/U) = \overline{\mathrm{NE}}(X/U)_{K_X+\Delta+A\geq 0} + \sum_{i=1}^r \mbR_{\geq 0}[C_i].
\]
\end{theorem} 

Secondly, a variant of the base point free theorem for threefolds is known in which the contraction is a map onto an algebraic space (see \cite[Theorem 0.5]{keel99}, cf.\ \cite[Theorem 5.2]{hx13}) and in fact the full base point free theorem for threefolds is valid when $k=\Fp$.
\begin{definition} We say that a dlt pair $(X,\Delta)$ has \emph{normal divisorial centres up to a universal homeomorphism} if for every irreducible divisor $D \subseteq \lfloor \Delta \rfloor$, the normalisation $f \colon \tilde D \to D$ is a universal homeomorphism.
\end{definition}

\begin{proposition}[{\cite[Lemma 2.4]{GNT06} cf.\ \cite[Theorem 5.4]{hx13}}] \label{prop:contractions} Let $(X,\Delta)$ be a three-dimensional $\mbQ$-factorial dlt pair admitting a projective morphism $\pi \colon X \to U$ to a  scheme $U$ defined over a perfect field of characteristic $p>0$. Assume that $(X,\Delta)$ has normal divisorial centres up to a universal homeomorphism, and let $R$ be an extremal ray of $K_X+\Delta$ over $U$ such that $R \cdot S < 0$ for an irreducible divisor $S \subseteq \lfloor \Delta \rfloor$. Then the contraction $X\to Z$ of $R$ exists as a projective morphism of normal varieties,  and $\rho(X/Z)=1$.
\end{proposition}

\subsection{$W\mcO$-rationality and normality of divisorial centres}
In this subsection we recall some definitions and results from \cite{GNT06}. First, given a $k$-scheme $X$, over a perfect field $k$ of positive characteristic $p>0$, we define the sheaf $W(\mcO_X)$ as follows
\[
W(\mcO_X)(U) \defeq W(\mcO_X(U)) \text{ for an open subset } U \subseteq X.
\] 
		Here $W \colon \mathrm{Alg}_{\mathbb{F}_p} \to \mathrm{Alg}_{\mbZ}$ is the Witt vector functor (see \cite[Subsection 2.5]{GNT06}).

Let $\mcA(X)$ be the abelian category of sheaves of abelian groups on $X$, and let $\mcA(X)_{\mbQ}$ be the quotient category
\[
\mcA(X)_{\mbQ} \defeq \mcA(X) \, / \, \{\mcF \in \mcA(X) \mid n\mcF = 0 \text{ for some } n \in \mbZ_{>0} \}.
\]
We set $W\mcO_{X,\mbQ}$ to be the image of $W\mcO_X$ in $\mcA(X)_{\mbQ}$ under the natural projection. Given a morphism of $k$-schemes $f \colon Y \to X$, one can define derived functors $R^if_*(W\mcO_{Y,\mbQ})$ for $i\geq 0$.
\begin{definition} We say that a $k$-variety $X$ for a perfect field $k$ of positive characteristic has \emph{$W\mcO$-rational singularities}, if there exists a (equivalently, for any) resolution of singularities $\phi \colon Y \to X$ such that $R^i\phi_*(W\mcO_{Y,\mbQ}) = 0$ for $i>0$.
\end{definition}

The following result asserts that normality of divisorial centres and $W\mcO$-rationality are preserved under pl-contractions.
\begin{proposition}[{\cite[Proposition 3.4 and 3.11]{GNT06}}] \label{prop:GNT}
Let $(X,\Delta)$ be a $\mbQ$-factorial three-dimensional dlt pair defined over a perfect field $k$ of characteristic $p>0$, with normal divisorial centres up to a universal homeomorphism. Let $g \colon X \to Z$ be a projective birational morphism onto a normal variety $Z$ such that
\begin{itemize}
	\item $-(K_X+\Delta)$ is $g$-ample,
	\item $-S$ is $g$-ample for an irreducible component $S$ of $\lfloor \Delta \rfloor$.
\end{itemize}
Then the irreducible components of $\lfloor g_*\Delta \rfloor$ are normal up to a universal homeomorphism. Moreover, if $X$ has $W\mcO$-rational singularities, then so does $Z$.
\end{proposition}

Lastly, let us state the following result which shows that $W\mcO$-rationality and normality of divisorial centres up to a universal homeomorphism {are} preserved by curve extractions. The proposition will be used to study the behaviour of these notions under flips as in \cite{GNT06}.
\begin{proposition} \label{prop:ascend_witt_rationality} Let $k$ be a perfect field of characteristic $p>0$, and let $g \colon Y \to Z$ be a proper birational morphism between normal threefolds over $k$ such that every fibre of $g$ is at most one-dimensional. Assume that $(Y, \Delta)$ is dlt for some effective $\mbQ$-divisor $\Delta$. Then
\begin{itemize}
	\item if $R^1g_*(W\mcO_{Y,\mbQ}) = 0$ and irreducible components of $\lfloor g_*\Delta \rfloor$ are normal up to a universal homeomorphism, then $(Y,\Delta)$ has normal divisorial centres up to a universal homeomorphism;
	\item if $Z$ has $W\mcO$-rational singularities, then so does $Y$. 
\end{itemize} 
\end{proposition}
\begin{proof}
This follows from \cite[Proposition 3.8]{GNT06}, \cite[Proposition 3.12]{GNT06}, and \cite[Lemma 2.5]{GNT06}.
\end{proof}
Note that if $R^1g_*\mcO_Y = 0$, then $R^1g_*(W\mcO_{Y,\mbQ}) = 0$ (\cite[Lemma 2.19 and Lemma 2.21]{GNT06}).

\section{The proof of Theorem \ref{thm:main}}
In this section we prove Theorem \ref{thm:main}. We begin with the following observation.

\begin{lemma} \label{lemma:extremal_rays} Let $(Y,\Delta)$ be a $\mbQ$-factorial  dlt pair defined over a field $k$. 
Assume that there exists a commutative diagram of projective birational morphisms
\begin{center}
\begin{tikzcd}
Y \arrow{d}{\pi} \arrow{r}{g} & Z \arrow{ld}  \\
X &
\end{tikzcd}
\end{center}
such that $X$ is a normal $\mbQ$-factorial variety, $\Exc(\pi) \subseteq \lfloor \Delta \rfloor$, and $g$ is a small birational morphism. Let $R$ be a $g$-exceptional curve. Then there exists an irreducible effective $\pi$-exceptional divisor $E'$ on $Y$ satisfying $R \cdot E' > 0$.
\end{lemma}
\begin{proof}
Write $\Delta = E + B$, where $B\geq 0$ and  $E = \sum _{i=1}^kE_i$ is the union of all irreducible $\pi$-exceptional divisors $E_i$. Assume by contradiction, that the statement of the lemma is false, that is  $R \cdot E_i \leq 0$ for every $i \geq 1$. Since $X$ is $\mbQ$-factorial, there exists an effective $\pi$-exceptional $\pi$-antiample divisor on $Y$, and hence $R$ is negative on some irreducible exceptional divisor. {Thus, without loss of generality,} we may assume $R \cdot E_1 < 0$, and so $R \subseteq E_1$.

{Set $L \defeq g^*H$ for a very ample divisor $H$ on $Z$ such that { $\Supp L$} and $E$ have no common components.} Since $g_*E_i\ne 0$ for $i=1,\ldots, k$, there exists a  $\pi$-relative curve $C$ on $E_1$, not contained in $\bigcup_{i \geq 2} E_i$, such that $L \cdot C > 0$. 

Given that $X$ is $\mbQ$-factorial, by pushing forward $L$ to $X$ and pulling it back, we obtain
\[
L \sim_{\pi, \mbQ} \sum_{i\geq 1} -b_i E_i
\]
for some $b_i \geq 0$. Since $L \cdot C >0$ and $E_i \cdot C \geq 0$ for $i \geq 2$, we get $E_1 \cdot C < 0$ and $b_1 > 0$. {Therefore},
\[
L \cdot R = -b_1E_1 \cdot R - \sum_{i\geq 2} b_iE_i \cdot R \geq -b_1E_1 \cdot R > 0,
\] 
which is a contradiction since $L\cdot R=0$ by our choice of $R$ and $L$; here we used the ad absurdum assumption that $R \cdot E_i \leq 0$ for all $i\geq 1$.
\end{proof}

{
\begin{remark} \label{remark:contractions_over_algebraic_spaces}
The above lemma holds for $X$ being a normal integral separated $\mbQ$-factorial algebraic space of finite type over $k$ (see \stacksproj{083Z}). Here, we call an algebraic space $\mathbb{Q}$-factorial if every Weil divisor is Cartier (that is Cartier up to a surjective \'etale cover by a scheme, cf.\ \stacksproj{083C}). The key case we will use in this article is when $X$ admits a {projective} birational morphism $\phi \colon X' \to X$ such that $X'$ is a normal $\mbQ$-factorial variety, $\rho(X'/X)=1$, and $\Exc(\phi)$ is a divisor. To verify that $X$ is $\mbQ$-factorial, we can replace it by an \'etale cover and assume that it is a variety, in which case the $\mbQ$-factoriality follows by a standard argument.
\end{remark}}
\subsection{Flips}
In this subsection, we tackle the existence of flips. {The following lemma is a key ingredient in the proof of Proposition \ref{prop:flips_exist}.}

\begin{lemma} \label{lemma:rel_surfaces} Let $(S,C + B)$ be a two-dimensional plt pair defined over an infinite perfect field $k$ of characteristic $p>0$, where $C$ is an irreducible curve, and let $f \colon S \to T$ be a projective birational morphism onto a surface germ $(T,0)$ such that $-(K_S+C+B)$ {and $C$ are $f$-nef.} Then $(S,C+B)$ is relatively purely F-regular over a neighbourhood of $f(C) \subseteq T$.
\end{lemma}
\begin{proof} By \cite[Remark 2.6]{HW17} and Stein factorisation, we can assume that $T$ is normal. {By the relative base point free theorem, $-(K_S+C+B)$ is semi-ample over $T$, and hence there exists an effective $\mbQ$-divisor $D \sim_{\mbQ} -2(K_S+C+B)$ such that $(S,C+B+D)$ is log canonical (this follows from \cite[Theorem 1]{tanaka_relative17} after having taken an lc compactification as in the proof of \cite[Proposition 2.10]{hnt}). Replacing $B$ by $B+\frac{1}{2}D$, we can suppose that $-(K_S+C+B) \sim_{f, \mbQ} 0$ and $(S,C+B)$ is plt.} 

Let $E \defeq f_*C$ and $B_T \defeq f_*B$. Since $C$ is irreducible and $f$-nef, one sees that $E\ne 0$ is an irreducible curve. {Moreover, we have that $K_S+C+B = f^*(K_T + E + B_T)$,} and so $(T, E+B_T)$ is plt. In particular, $E$ is normal, and $(E , B_{E})$ is klt where $K_{E}+B_{E}=(K_T+E+B_T)|_{E}$. Since $\dim E=1$, $(E , B_{E})$ is log smooth and so it is strongly F-regular. {Consequently,} F-adjunction shows that  $(T,E+B_T)$ is purely F-regular (see for instance Lemma \ref{lemma:relative-inversion-of-adjunction}), and \cite[Proposition 2.11]{hx13} implies that $K_S+C+B$ is relatively purely F-regular.
\end{proof}

The following proposition in conjunction with Lemma \ref{lemma:extremal_rays} gives the existence of flips as in Theorem \ref{thm:main}.
\begin{proposition} \label{prop:flips_exist} Let $(X,\Delta)$ be a three-dimensional $\mbQ$-factorial dlt pair defined over a perfect field $k$ of characteristic $p>0$. Let $\phi \colon X \to Z$ be a flipping contraction of a $(K_X+\Delta)$-negative extremal ray $R$. Suppose that there exist irreducible divisors $E, E' \subseteq \lfloor \Delta \rfloor$ such that $R \cdot E < 0$ and $R \cdot  E' > 0$. Then the flip $(X^+, \Delta^+)$ of $f$ exists. 

Moreover, if $(X,\Delta)$ has normal divisorial centres up to a universal homeomorphism, then so does $(X^+, \Delta^+)$.
\end{proposition}

\begin{remark} Note that if $(X,S + S' + B)$ is a three-dimensional dlt pair such that $S$, $S'$ are distinct irreducible divisors and $\lfloor B \rfloor = 0$, then $(\tilde S, \Delta _{\tilde S})$ is plt, where $\tilde S$ is the normalisation of $S$ and  $K_{\tilde S} + \Delta _{\tilde S} = (K_X +S+S'+B)|_{\tilde S}$. This can be verified on a log resolution of $(X,S+S'+B)$.

\end{remark}
\begin{proof} We work over a neighborhood of the closed point $P=\phi (R)\subset Z$ and hence we will frequently replace $Z$ by an appropriate neighborhood of $P\in Z$. By base change to its algebraic closure, we can assume that $k$ is algebraically closed. Except for $\mbQ$-factoriality, all the assumptions are preserved under this base change (cf.\ \cite[Remark 2.7(1)]{GNT06}). Moreover, it is enough to construct a flip after this base change by \cite[Remark 2.7(2)]{GNT06}. Although $X$ may not be $\mbQ$-factorial, the irreducible components of $\Delta$ are $\mbQ$-Cartier and this is all that we need in the following proof.

Write $\Delta = E + E' + D$, where $D \geq 0$, and set 
\[
\Delta^{\mathrm{plt}} \defeq E + (1-\epsilon)(E'+D)
\]
for some $\epsilon > 0$ sufficiently small so that $(K_X + \Delta^{\mathrm{plt}}) \cdot R < 0$ { and $(K_X + \Delta^{\mathrm{plt}} + \epsilon E') \cdot R < 0$.} 

We aim to show that $(X, \Delta^{\mathrm{plt}})$ is relatively purely F-regular over $Z$. To this end, write 
\begin{align*}
	K_{\tilde E} + \Delta_{\tilde E} &= (K_X + \Delta^{\mathrm {plt}})|_{\tilde E}, \text{ and }\\
	K_{\tilde E} + \Delta'_{\tilde E} &= (K_X + \Delta^{\mathrm {plt}} + \epsilon E')|_{\tilde E},
\end{align*}
where $g \colon \tilde E \to E$ is the normalisation of $E$. 
Since $(X, \Delta)$ is dlt, the above remark shows that $(\tilde E, \Delta'_{\tilde E})$ is plt, and so Lemma \ref{lemma:rel_surfaces} implies that it is, in fact, relatively purely F-regular over a neighborhood of $P \in \phi(E)$. Hence we may assume that $(\tilde E, \Delta_{\tilde E})$ is relatively F-regular.

By inversion of F-adjunction (see Lemma \ref{lemma:relative-inversion-of-adjunction}), this implies that $(X, \Delta^{\mathrm{plt}})$ is relatively purely F-regular over $Z$. In particular, the proof of the existence of the flip of $(X, \Delta^{\mathrm{plt}})$ from \cite[Thereom 4.12]{hx13} holds without any change (see Remark \ref{remark:hx}). Since such a flip is also the flip of $(X,\Delta)$, we obtain $(X^+, \Delta^+)$ sitting inside the following diagram
\begin{center}
\begin{tikzcd}
X \arrow[swap]{rd}{\phi} \arrow[dashed]{rr}{\psi} & & X^+ \arrow{ld}{\phi^+} \\
& Z, &
\end{tikzcd}
\end{center}
where $\Delta^+ \defeq \psi_* \Delta$. 

In order to show the normality of divisorial centres of $\lfloor \Delta^+ \rfloor$ up to a universal homeomorphism, we proceed as follows. First, we have that irreducible components of $\phi_*\lfloor \Delta \rfloor$ are normal up to a universal homeomorphism by Proposition \ref{prop:GNT}. 

By \cite[Corollary 6.4]{schwedesmith10}, since $X$ is {relatively} F-regular, we have that $X^+$ is relatively F-regular as well. Pick an ample divisor $A$ such that $R^1\phi^+_* \mcO_{X^+}(A) = 0$. Since $F^e \colon \mcO_{X^+} \to F^e_* \mcO_{X^+}(A)$ splits locally over $Z$ for some $e>0$, we get that
\[
R^1\phi^+_*\mcO_{X^+} \subseteq F^e_* R^1\phi^+_*\mcO_{X^+}(A) = 0.
\]
In particular, $R^1\phi^+_*(W\mcO_{X^+, \mbQ})=0$ by \cite[Lemma 2.19 and Lemma 2.21]{GNT06}. By Proposition \ref{prop:ascend_witt_rationality}, this implies that $(X^+,\Delta^+)$ has normal divisorial centres up to a universal homeomorphism.
\end{proof} 

{
\begin{remark} \label{remark:hx} In the proof of the above theorem, we used the fact that if $\phi \colon X \to Z$ is a $(K_X+S+B)$-flipping contraction of an extremal curve $R$ for a three-dimensional relatively purely F-regular pair $(X,S+B)$ defined over {a perfect} field of characteristic $p>0$ with $S$ being an irreducible divisor such that $R \cdot S < 0$, then the flip of $\phi$ exists. This has been implicitly proven in \cite[Theorem 4.12]{hx13}. 

{The statement of \cite[Theorem 4.12]{hx13} assumes that the coefficients of $\Delta^{\mathrm{plt}}$ are standard, but this is only used to show that $(X, S+B)$ is relatively purely F-regular, which we assume apriori to be true. Moreover, the results of \cite{hx13} are stated over an algebraically closed field, but it is enough to show the existence of flips after base-changing to the algebraic closure $\overline{k}$ of $k$ (see \cite[Remark 2.7(2)]{GNT06}). Here note that the $\mbQ$-factoriality of $X$ need not be preserved, but fortunately $\mbQ$-factoriality is not needed in the proof of \cite[Theorem 4.12]{hx13}.}

For the convenience of the reader, we recall the general strategy of the proof.

First, by \cite[Proposition 4.1]{hx13} we have that $S$ is normal. Now, fix an effective Cartier divisor $Q \sim k(K_X+S+B)$ on $X$ for some $k>0$ sufficiently divisible such that the support of $Q$ does not contain $S$, and define the following b-divisors
\[
\mathbf{N}_i \defeq \mathbf{Mob}(iQ), \quad \mathbf{M}_i \defeq \mathbf{N}_i|_S, \ \text{ and } \ \mathbf{D}_i \defeq \frac{1}{i}\mathbf{M}_i.
\]
Write $K_S + B_S = (K_X+S+B)|_S$, and let $\bar S$ be the terminalisation of $(S,B_S)$. Then $\mathbf{M}_i$ descends to $\bar S$ by \cite[Lemma 4.4]{hx13}.

Following the strategy depicted in \cite{cortibook}, it is enough to show that $\mathbf{D}_i$ stabilise for $i\gg 0$. To this end, we note that the $\mathbb{R}$-divisor $\mathbf{D} \defeq \lim \mathbf{D}_i$ is semiample on $\bar S$ by \cite[Lemma 4.8]{hx13}. Let $a \colon \bar S \to S^+$ be the induced fibration as in \cite[Discussion after Lemma 4.8]{hx13}. By \cite[Corollary 4.11]{hx13}, we have that $\mathbf{D}_{\bar S}$ is rational and $a_*\mathbf{D}_{\bar S} = a_*\mathbf{D}_{j, \bar S}$ for divisible enough $j\gg 0$.

Let $g \colon Y \to X$ be an appropriately chosen resolution of singularities such that $ \mathbf{N}_{i,Y}$ is free. We replace $Y$ upon changing $i$. Set $L_{i,j}  \defeq \lceil \frac{j}{i} \mathbf{N}_{i,Y} + \mathbf{A}_Y \rceil$ for $i,j>0$ and $\mathbf{A}$ being the discrepancy b-divisor for $(X,S+B)$. Explicitly:
\[
K_Y + S' = f^*(K_X + S + B) + \mathbf{A}_Y,
\]
where $S'$ is the strict transform of $S$. Since $(S,B_S)$ is relatively purely F-regular, the proof of \cite[Lemma 4.13]{hx13} holds without any change, and so
\begin{equation} \label{eq:frob_surj}
S^0(S', \sigma(S', \psi_{S'}) \otimes \mcO_{S'}(L_{i,j}|_{S'})) = H^0(S^+, \mcO_{S^+}(j\mathbf{D}_{S^+}))
\end{equation}
for $j$ replaced by some multiple and $i \gg 0$ divisible by $j$, where the left hand side is the subset of sections in $H^0(S', \mcO_{S'}(L_{i,j}|_{S'}))$ which are stable under the Frobenius trace map of $(S', \psi_{S'})$ (see \cite[Section 2.3]{hx13} for the precise definition). Here $\psi_{S'} \defeq (\psi - S')|_{S'}$ and $\psi$ is a small perturbation of $\{-\frac{j}{i}\mathbf{N}_{i,Y} - \mathbf{A}_Y\}+S'$ (see \cite[Lemma 4.6]{hx13}). Moreover, the natural morphism
\begin{align*}
S^0(Y, \sigma(Y,\psi)\otimes \mcO_Y(L_{i,j})) \to \ & S^0(S', \sigma(S', \psi_{S'}) \otimes \mcO_{S'}(L_{i,j}|_{S'})) \\ &= H^0(S^+, \mcO_{S^+}(j\mathbf{D}_{S^+}))
\end{align*}
is surjective (see \cite[Lemma 4.7]{hx13}). Since
\[
|L_{i,j}| \subseteq |jk(K_Y + S' + B_Y)| + \lceil A_Y \rceil,
\]
where $B_Y = (-A_Y)_{\geq 0}$, this allows for lifting sections and showing the stabilisation of $\mathbf{D}_i$ (see the end of the proof of \cite[Theorem 4.12]{hx13} for details).  

To sum up, the relative pure F-regularity of $(X,S+B)$ is used twice: to show the normality of $S$, and, more importantly, to establish the {identity} (\ref{eq:frob_surj}) which is fundamental in lifting sections. 
\end{remark}}

\subsection{Proof of Theorem \ref{thm:main} and Theorem \ref{thm:normality}}

For inductive reasons and the sake of clarity, we tackle Theorem \ref{thm:main}, Theorem \ref{thm:normality}, and  Corollary \ref{cor:wittrationality} simultaneously. Using the following proposition we will show that every  three-dimensional dlt pair has normal divisorial centres up to a universal homeomorphism, which in turn will conclude the proof of Theorem \ref{thm:main}.{ This approach follows, mutatis mutandis, the strategy laid down in \cite[Theorems 3.14, 3.15, 3.16]{GNT06}.}
\begin{proposition} \label{prop:main} Let $(Y,\Delta)$ be a $\mbQ$-factorial three-dimensional dlt pair defined over a perfect field of characteristic $p>0$ with normal divisorial centres up to a universal homeomorphism. Assume that there exists a projective morphism $\pi \colon Y \to X$ over a normal $\mbQ$-factorial variety $X$ such that $\Exc(\pi) \subseteq \lfloor \Delta \rfloor$. 

Then we can run a {$(K_Y+\Delta)$-MMP} over $X$ which terminates with a minimal model $(Z,\Delta_Z)$, whose divisorial centres are normal up to a universal homeomorphism. Moreover, if $Y$ has $W\mcO$-rational singularities, then so does $Z$.
\end{proposition}
\begin{proof}
The cone theorem is valid by Theorem \ref{thm:cone}.  Since $X$ is $\mbQ$-factorial, there exists a relatively anti-ample effective exceptional divisor. In particular, there exists $S \subseteq \lfloor \Delta \rfloor$ such that $S \cdot R < 0$ for a given $(K_X+\Delta)$-extremal ray $R$. Proposition \ref{prop:contractions} implies the existence of the contraction $g \colon X \to Z$ of $R$. By Proposition \ref{prop:GNT}, the normality of divisorial centres up to a universal homeomorphism is preserved. By Lemma \ref{lemma:extremal_rays} and Proposition \ref{prop:flips_exist}, if $g$ is a flipping contraction, then the flip exists and the normality of divisorial centres up to a universal homeomorphism is preserved {under $g$}. The termination of the MMP with scaling holds by standard arguments (cf.\ \cite{hx13} and \cite{BW14}). Thus, a minimal model exists.

Each step of the MMP preserves the $W\mcO$-rationality of singularities by Proposition \ref{prop:GNT} and Proposition \ref{prop:ascend_witt_rationality}. Therefore, if $Y$ has $W\mcO$-rational singularities, then so does $Z$.
\end{proof}

In order, to show Theorem \ref{thm:main}, we need to prove Theorem \ref{thm:normality} first.
\begin{proof}[Proof of Theorem \ref{thm:normality}]
Let $\pi \colon Y \to X$ be a log resolution of $(X, S+B)$. By Proposition \ref{prop:main}, we can run a $(K_Y+\pi^{-1}_*(S+B) + \Exc(\pi))$-MMP over $X$. The corresponding minimal model must be $(X,S+B)$ itself. In fact the plt condition guarantees that all the components of $\Exc(\pi)$ are contracted and hence the minimal model is isomorphic to $X$ in codimension one. Since $X$ is $\mbQ$-factorial, the above minimal model is in fact isomorphic to $X$. By Proposition \ref{prop:main}, the normalisation of $S$ is a universal homeomorphism. 
\end{proof}

\begin{proof}[Proof of Theorem \ref{thm:main}]
By Theorem \ref{thm:normality} we have that the divisorial centres of $(Y,\Delta)$ are normal up to a universal homeomorphism. The theorem now follows from Proposition \ref{prop:main}.
\end{proof}

Corollaries \ref{cor:wittrationality}, \ref{cor:dlt_mod} and \ref{cor:plt} now follow.
\begin{proof}[Proof of Corollary \ref{cor:wittrationality}]
Let $\pi \colon Y \to X$ be a log resolution of $X$. By Proposition \ref{prop:main}, we can run a $(K_Y+\Exc(\pi))$-MMP over $X$, the minimal model of which must be $X$ itself, as it is klt and $\mbQ$-factorial. In particular, $X$ has $W\mcO$-rational singularities by Proposition \ref{prop:main}.
\end{proof}

\begin{proof}[Proof of Corollary \ref{cor:dlt_mod}]
Let $\pi \colon Y \to X$ be a log resolution of $(X, \Delta)$. Then, a dlt modification is a minimal model of $(Y,\pi^{-1}_*\Delta + \Exc(\pi))$ over $X$ (see Theorem \ref{thm:main}).
\end{proof}

\begin{proof}[Proof of Corollary \ref{cor:plt}]
By taking a log resolution of $(X,S+B)$ it is easy to see that if $(X,S+B)$ is plt, then $(\tilde S,B_{\tilde S})$ is klt. Thus, we can assume that $(\tilde S,B_{\tilde S})$ is klt and aim to show that $(X,S+B)$ is plt near $S$. 

Let $\pi \colon Y \to X$ be a dlt modification of $(X,S+B)$ (see Corollary \ref{cor:dlt_mod}) and write $K_Y + S_Y + B_Y = \pi^*(K_X+S+B)$. {By definition} (of a dlt modification) for any $\pi$-exceptional irreducible divisor $E$ we have that $E \subseteq \lfloor B_Y \rfloor$.  Write
\[
(\pi|_{S_Y})^*(K_{\tilde S} + B_{\tilde S}) = (K_Y+S_Y+B_Y)|_{\tilde S_Y} = K_{\tilde S_Y} + B_{\tilde S_Y}, 
\]
where $\tilde S_Y \to S_Y$ is the normalisation of $S_Y$, and $B_{\tilde S_Y}$ is the different. Let $E$ be a $\pi$-exceptional divisor intersecting $S_Y$. Since $E \subseteq \lfloor B_Y \rfloor$ and $(Y,S_Y+\Exc(\pi))$ is dlt, we must have that $E \cap S_Y \subseteq \lfloor B_{\tilde S_Y} \rfloor$. {If $E \cap S_Y \neq \emptyset$, then} this contradicts $(\tilde S, B_{\tilde S})$ being klt. 

Therefore we may assume that $E\cap S_Y=\emptyset$ so that $Y=X$ near $S$ and hence $(X,S+B)$ is dlt on a neighborhood of $S$. Since $S$ is irreducible, $(X,S+B)$ is in fact plt.
\end{proof}

\begin{remark} \label{rem:mmp_over_alg_spaces} Assume that $Y$ is projective. Then Theorem \ref{thm:main} holds in a slightly more general setting: suppose that $X$, instead of being a normal $\mbQ$-factorial variety, is an algebraic space admitting a birational morphism $\phi \colon X' \to X$ such that $X'$ is a normal $\mbQ$-factorial variety, $\rho(X'/X)=1$, and $\Exc(\phi)$ is a divisor. In view of Remark \ref{remark:contractions_over_algebraic_spaces}, the assertions of Theorem \ref{thm:main} hold true by exactly the same proof. Here, the assumption on the projectivity of $Y$ is needed in order to apply the cone theorem.
\end{remark}

\begin{remark} \label{rem:smaller_centres_normal} In fact all log canonical centres of a three-dimensional $\mbQ$-factorial dlt pair $(X,\Delta)$ are normal up to a universal homeomorphism. Indeed, if the centre is two-dimensional, then this follows from Theorem \ref{thm:main}. Thus, we can assume that we have a dlt pair $(X,D_1+D_2)$ and we need to show that every irreducible component of $D_1 \cap D_2$, say $C$, is normal up to a universal homeomorphism. Let $u \colon D^n_1 \to D_1$ be the normalisation of $D_1$ and let $C' \defeq u^{-1}(C)$. By adjunction (cf.\ Corollary \ref{cor:plt}), $(D_1^n, \mathrm{Diff})$ is plt, where $K_{D^n_1} + \mathrm{Diff} = (K_X+D_1+D_2)|_{D^n_1}$. By the surface theory, $C' \subseteq \lfloor \mathrm{Diff} \rfloor$ is smooth. In particular, the universal homeomorphism $u|_{C'} \colon C' \to C$ is also the normalisation. 
\end{remark}

\section{Minimal Model Program for dlt reductions}
In this section we show Theorem \ref{thm:mmp_reductions}. The main ingredient of the proof is the following result.
\begin{proposition} \label{prop:flips_for_reductions} Let $(X,S+B)$ be a three-dimensional $\mbQ$-factorial plt pair defined over a perfect field $k$ of characteristic $p>0$, where $S$ is an irreducible divisor. Let $g \colon X \to Z$ be a flipping contraction of a $K_X+S+B$ negative extremal curve $R$ such that $R \subseteq S$ and  $R \cdot S = 0$. Then the flip of $g$ exists.
\end{proposition}
The proof follows the same strategy as the reduction of the existence of flips to pl-flips (\cite{fujino05}). The necessary pl-flips exist by Proposition \ref{prop:flips_exist}. 
\begin{proof} It suffices to prove the claim in a neighborhood of the point $P :=g(R)\subset Z$. Since $S \cdot R = 0$, we have that $S' \defeq g(S)$ is $\mbQ$-Cartier. Set $\Delta = S+B$ and let $H'$ be a reduced Cartier divisor on $Z$  containing the  non simple normal crossing locus of $(Z, g_*\Delta)$ and satisfying the following {conditions} (see assumptions (i)-(iv) in \cite[Theorem 4.3.7]{fujino05}):
\begin{enumerate}
	\item $H \defeq g_*^{-1}H'$ contains $\Exc(g)$,
	\item for any proper birational morphism $h \colon Y \to Z$ such that $Y$ is $\mbQ$-factorial, we have that $N^1(Y/Z)$ is generated by the irreducible components of the strict transform of $H'$ and the $h$-exceptional divisors,
	\item $H$ and $\Delta$ have no irreducible components in common.
\end{enumerate}
Consider a log resolution of $(X,\Delta + H)$:
\begin{center}
\begin{tikzcd}
Y \arrow{r}{p} \arrow[bend left = 35]{rr}{h} & X \arrow{r}{g} & Z.
\end{tikzcd}
\end{center}
We claim that we can run a $(K_Y+\Delta_Y+H_Y)$-MMP over $Z$, where $H_Y$ is the strict transform of $H$, and $\Delta_Y \defeq p_*^{-1}\Delta + \Exc(h)$. 
First of all note that for any extremal ray $R$ for $Y$ over $Z$, we may assume that $R$ is contained in the support of $h^*H'$ and by condition (2) there is a component of the support of $h^*H'$ having a non-zero intersection number with $R$. Since $R\cdot h^*H'=0$, there are components  $E,E'$ of  the support of $h^*H'$ such that $E\cdot R<0$ and $E'\cdot R>0$. 
Divisorial centres are normal up to a universal homeomorphism by Theorem \ref{thm:normality}, and so contractions exist by Proposition \ref{prop:contractions}. The cone theorem holds by Theorem \ref{thm:cone}, and  special termination is valid by the standard argument (see \cite[Theorem 4.2.1]{fujino05}). 
The necessary flips exist by Proposition \ref{prop:flips_exist}.\\


Let us replace $(Y,\Delta_Y+H_Y)$ by its minimal model over $Z$. Note that, after this replacement, $Y$ need not admit a map to $X$, but it still admits a map to $Z$, which as above we denote by $h \colon Y \to Z$. Write $\Delta_Y = D + B_Y$, where $B_Y=h^{-1}_*g_* B$ so that $\lfloor B_Y \rfloor = 0$ and $D = \sum_{i=1}^m D_i$ is a sum of irreducible divisors not contained in $H_Y$ with $D_1 = S_Y$ being the strict transform of $S$ and $D_2,\ldots , D_m$ are the $h$-exceptional divisors. Since $h^*H' \equiv_h 0$, we have that
\[
H_Y \equiv_h -\sum_j b_j D_j, \text{ where } b_j \in \mbQ_{\geq 0}.
\]
Run a $(K_Y+\Delta_Y)$-MMP with scaling of $H_Y$. Let $0 < \lambda \leq 1$ be such that $K_Y+\Delta_Y + \lambda H_Y$ is $h$-nef, and there exists a $(K_Y+\Delta_Y)$-extremal curve $R$ satisfying $(K_Y+\Delta_Y+\lambda H_Y) \cdot R = 0$. Since $(K_Y+\Delta_Y) \cdot R < 0$, we have that $H_Y \cdot R > 0$, and the above numerical equivalence implies that there exists $j$ for which $D_j \cdot R < 0$.  Arguing as above, in order to show that such an MMP can be run, it is enough to show that flips exist. Hence we suppose that the associated contraction is flipping, and so in particular $h(R)=P$. We have $h^*S' \cdot R = 0$ and $g(R)=P$, and thus there exists $j'$ for which $D_{j'} \cdot R \geq 0$ and $D_{j'} \cap R \neq \emptyset$. We claim that $D_{j'} \cdot R > 0$. Indeed, otherwise $R \subseteq D_{j'}$, and since $(Y, D_j + D_{j'})$ is dlt, we get that $R = D_j \cap D_{j'}$. Then
\[
R \cdot D_{j'} = R|_{D_j} \cdot D_{j'}|_{D_j} = \lambda (R|_{D_j})^2 < 0,
\]
for some $\lambda >0$, which is a contradiction. Given that $D_j \cdot R < 0$ and $D_{j'} \cdot R > 0$, the flip of $R$ exists by Proposition \ref{prop:flips_exist}.

Now, replace $(Y,\Delta_Y)$ by the output of this MMP. Then $K_Y+\Delta _Y$ is nef over $X$.  Notice that $h$ is small (by the negativity lemma since $(X,S+B)$ was plt) and $(Y,\Delta_Y)$ is the flip of $(X,\Delta)$. It is then easy to see that  $\rho (Y/X)=1$ (cf.\ \cite[Lemma 1.6]{AHK07}) and so  $K_Y+\Delta _Y$  is in fact ample over $X$. 
\end{proof}  

{\begin{proof}[Proof of Theorem \ref{thm:mmp_reductions}]
Note that $X$ is not of finite type over $k$ and hence is not a variety. In what follows, when we invoke results about varieties, we implicitly refer to the base change of $(\mathcal{X}, \Phi)$ with respect to an appropriate open neighbourhood of $s \in C$.

We run a $(K_X+\Delta)$-MMP with scaling of $A$, a sufficiently ample divisor over $S \defeq {\rm Spec}(R)$. We need to show that the cone theorem and the base point free theorem are valid for $(X,\Delta)$, that the flips exist, and that the MMP terminates. Let $A$ be an ample $\mbQ$-Cartier $\mbQ$-divisor on $X$.\\

\textbf{Cone theorem} Let $\smash{X_s = \bigcup_{j=1}^m E_j}$ be the reduced special fibre of $\pi$, where $E_j$ are irreducible components, and let $\tilde E_j \to E_j$ be the normalisations. Take $\{\Gamma_{j,k}\}_{k \geq 0}$ to be the  finite collection of images of  
\[
K_{\tilde E_j} + \Delta_{\tilde E_j} + A|_{\tilde E_j} = (K_X+\Delta+A)|_{\tilde E_j}
\]
{negative extremal curves} in $\tilde E_j$. Since $R$ is a DVR, $\overline{\mathrm{NE}}(X/S)$ is generated by curves contained in $X_s$, and so       
\[
\overline{\mathrm{NE}}(X/S) = \overline{\mathrm{NE}}(X/S)_{K_X+\Delta + A \geq 0} + \sum_{j,k} \mbR[\Gamma_{j,k}].
\]


\textbf{Base point free theorem}
Suppose that $K_X+\Delta+A$ is nef and $(X,\Delta+A)$ is dlt. We aim to show that it is semiample. By \cite[Theorem 1.1]{CT17}), it is enough to check that $(K_X+\Delta+A)|_{X_s}$ and $(K_X+\Delta+A)|_{X_{\eta}}$ are semiample where $X_s$ and $X_{\eta}$ are the reduced special fibre, and the generic fibre, respectively. The latter $\mbQ$-divisor is semiample by the base point free theorem for surfaces (see \cite[Theorem 1.1]{tanakaimperfect}). Thus, it is enough to show that $(K_X+\Delta+A)|_{X_s}$ is semiample.

Let $\pi \colon W \to X_s$ be an $S_2$-fication (cf.\ \cite[Proposition 2.6]{waldronlc}) of $X_s$. By Remark \ref{rem:smaller_centres_normal}, all the log canonical centres of $(X,X_s)$ are normal up to a universal homeomorphism, and so  the proof of \cite[Proposition 5.5]{waldronlc} holds in our setting, showing that $\pi$ is a universal homeomorphism and $W$ is demi-normal. Therefore, $(W,\Delta_W)$ is slc, where $K_W + \Delta_W = \pi^*(K_X+\Delta)|_{X_s}$. 

By abundance for slc surfaces (see \cite[Theorem 0.1]{tanakaslc}), $K_W+\Delta_W$ is semiample, and, since $\pi$ is a universal homeomorphism, $(K_X+\Delta)|_{X_s}$ is semiample as well. \\

\textbf{Existence of flips} 
Let $g \colon X \to Z$ be a pl-flipping contraction of a $(K_X+\Delta)$-extremal curve $C$. In particular, $C$ is contained in the special fibre of $\phi \colon X \to S$, that is there exists an irreducible divisor  $E \subseteq \phi^{-1}(s)$ such that $C \subseteq E$. 

If $C \cdot E \neq 0$, then, given $C \cdot \phi^{-1}(s) = 0$, we get that there exists an irreducible divisor $E' \subseteq \phi^{-1}(s)$ satisfying $\mathrm{sgn}(C \cdot E') = -\mathrm{sgn}(C \cdot E)$. Thus the flip of $g$ exists by Proposition \ref{prop:flips_exist}. If $C \cdot E = 0$, then the flip exists by Proposition \ref{prop:flips_for_reductions}.\\


\textbf{Termination} All the flipped curves are contained in $\Supp \phi^{-1}(s) \subseteq \lfloor \Delta \rfloor$, and so the MMP terminates by special termination (see \cite{fujino05}).
\end{proof}}

\section{Other applications}{}
In this section, we gather some other results on three-dimensional birational geometry  that can be extended to low characteristic.
\begin{proposition}[{cf.\ \cite[Theorem 5.6]{hx13}}] \label{prop:divisorial_contractions} Let $(X,\Delta)$ be a projective $\mbQ$-factorial three-dimensional dlt pair defined over a perfect field $k$ of positive characteristic $p>0$. {Let $R$ be a $(K_X+\Delta)$-negative extremal ray.} Assume that the closure of the locus of curves which are numerically equivalent to {a multiple} of $R$ is two-dimensional. Then the contraction of $R$ exists.
\end{proposition}
\begin{proof}
{Given Remark \ref{rem:mmp_over_alg_spaces}, the proof is exactly the same as that of \cite[Theorem 5.6]{hx13}.}
\end{proof}

We can also generalise {the main results of} \cite{xuzhang18}.

\begin{proposition}[{cf.\ \cite[Theorem 3.2]{xuzhang18}}] \label{prop:tame_fundamental_group} Let $(X,x)$ be a three-dimensional $\mbQ$-factorial klt singularity defined over a perfect field of characteristic $p>0$. Then the tame fundamental group $\pi^{\mathrm{tame}}_1(X,x)$ is finite.
\end{proposition}
\begin{proof}
It follows by exactly the same proof as that of \cite[Theorem 3.2]{xuzhang18}. The reason for that proof to require the assumption $p>5$, was the use of the existence of a Koll\'ar component, which now follows from Theorem \ref{thm:main} and the standard argument.   
\end{proof}

{
\begin{proposition}[{cf.\ \cite[Theorem 3.6]{hacondas}}] \label{prop:lccentres} Let $(X,\Delta)$ be a $\mbQ$-factorial three-dimensional log canonical pair defined over a perfect field of characteristic $p>0$ and such that $X$ has klt singularities. Then minimal log canonical centres of $(X,\Delta)$ are normal up to a universal homeomorphism.
\end{proposition}
\begin{proof}
Let $W$ be a minimal log canonical centre of $(X,\Delta)$. By the tie-breaking trick, up to perturbing $\Delta$, we can assume that $W$ is a unique log canonical centre with a unique divisor over $X$ of discrepancy $-1$ (cf.\ the proof of \cite[Theorem 3.6]{hacondas}). If $\codim_X(W) = 1$, then $(X,\Delta)$ is plt, and the proposition follows by Theorem \ref{thm:normality}. Thus, we can assume that $\codim_X(W)=2$.

Let $f \colon Y \to X$ be a dlt modification of $(X,\Delta)$. By construction, $f$ extracts exactly one divisor $S$ and we have that 
\[
K_Y + S + \Delta' = f^*(K_X+\Delta),
\]
where $\Delta' \defeq f^{-1}_*\Delta$. Since $(Y,S+\Delta')$ is plt, $S$ is normal up to a universal homeomorphism, and so there exists an integer $e>0$ and a factorisation 
\[
\mcO_S \to \mcO_{S^n} \to F^e_*\mcO_S,
\]
 where $S^n$ is the normalisation of $S$. The proof of \cite[Theorem 3.5]{hacondas} does not require any assumptions on the characteristic (it only needs $S$ to be normal in codimension one which follows by means of localisation and surface theory), hence $R^1f_*\mcO_Y(-S)=0$ and the sequence
\[
0 \to f_*\mcO_Y(-S) \to \mcO_X \to f_*\mcO_S \to 0
\]
is exact. Note that $f_*\mcO_Y(-S) = \mcI_W$, where $\mcI_W$ is the ideal sheaf of $W$. 

The surjective morphism on the right factors as $\mcO_X \to \mcO_W \to f_*\mcO_S$, which implies that $\mcO_W \simeq f_*\mcO_S$. Thus, we have a sequence of injective morphisms
\[
\mcO_W \to \mcO_{W^n} \to f_*\mcO_{S^n} \to f_*F^e_*\mcO_S \simeq F^e_* \mcO_W.
\]
In particular, the normalisation of $W$ factors through the Frobenius $F^e$, and so it is a universal homeomorphism.
\end{proof}}
{One can also show that \cite[Lemma 2.1]{hacondas} holds for $p>0$. The only issue is that the techniques of our paper do  not allow for extracting arbitrary divisors with discrepancies smaller than zero. However, by using the tie-breaking trick as above, the requisite divisors in the proof of \cite[Lemma 2.1]{hacondas} can be extracted with the help of dlt modifications (see Corollary \ref{cor:dlt_mod}).}

\section*{Acknowledgements}{}
We would like to thank {Fabio Bernasconi}, Paolo Cascini, Hiromu Tanaka, Joe Waldron, {Chenyang Xu, and Lei Zhang} for comments and helpful suggestions.

The first author was partially supported by NSF research grants no: DMS-1300750, DMS-1840190, DMS-1801851 and by a grant
from the Simons Foundation; Award Number: 256202. He would also like
to thank the Mathematics Department and the Research Institute for Mathematical Sciences,
located Kyoto University and the Mathematical Sciences Research Institute in Berkeley were some of this research was conducted. The second author was supported by the Engineering and Physical Sciences Research Council [EP/L015234/1] during his PhD at Imperial College London, by the National Science Foundation under Grant No.\ DMS-1638352 at the Institute for Advanced Study in Princeton, and by the National Science Foundation under Grant No.\ DMS-1440140 while the author was in residence at the Mathematical Sciences Research Institute in Berkeley, California, during the Spring 2019 semester.

\bibliographystyle{amsalpha}
\bibliography{final}

\end{document}